\documentclass[11pt]{article}

\usepackage{amsmath,amsfonts,amssymb,amsthm,tikz}

\title{Kochen-Specker sets in four-dimensional spaces}

\author{
Brandon Elford\\
Petr Lison\v{e}k\thanks{Corresponding author.
E-mail: plisonek@sfu.ca}
\\
\ \\
Department of Mathematics\\
Simon Fraser University\\
Burnaby, BC, V5A 1S6\\
Canada
}

\newtheorem{theorem}{Theorem}[section]
\newtheorem{lemma}[theorem]{Lemma}
\newtheorem{proposition}[theorem]{Proposition}

\theoremstyle{definition}
\newtheorem{definition}[theorem]{Definition}

\def\Z{{\mathbb Z}}
\def\C{{\mathbb C}}
\def\R{{\mathbb R}}

\def\V{{\mathcal V}}
\def\B{{\mathcal B}}

\def\z{\zeta}

\begin{document}

\date{}

\maketitle

\begin{abstract}
For the first time we construct an infinite family of Kochen-Specker
sets in a space of fixed dimension, namely in $\R^4$. 
While most of the previous constructions of Kochen-Specker
sets have been based on computer search,
our construction is analytical and it comes with a short,
computer-free proof.
\end{abstract}

\section{Introduction}
The Kochen-Specker theorem (KS~theorem)
is an important result in quantum mechanics 
\cite{KS-paper}.
It demonstrates the contextuality of quantum mechanics,
which is one of its properties that may become crucial
in quantum information theory \cite{How}.
In this paper we focus 
on proofs of the KS theorem that are given by showing
that, for $n\ge 3$, 
there does not exist a function $f:\C^n\rightarrow \{0,1\}$
such that for every orthogonal basis $B$ of $\C^n$ there
exists {\em exactly one} vector $x\in B$ such that $f(x)=1$,
where $\C^n$ denotes the $n$-dimensional vector space over the field
of complex numbers.
(The function $f$ is sometimes called a ``two-valued state.'')
This particular approach has been used in many publications,
see for example \cite{Adan-18,Lis-PRA,WA-preprint} 
and many references cited therein.
Definition \ref{def-KS-pair} 
given below formalizes 
one common way of constructing such proofs,
using a simple parity argument.
Therefore the structures that satisfy 
Definition \ref{def-KS-pair} are sometimes
referred to as {\em parity proofs} of the KS~theorem.

\begin{definition}
\label{def-KS-pair}
We say that $(\V,\B)$ is a {\em Kochen-Specker pair in $\C^n$}
if it meets the following conditions:
\begin{itemize}
\item[(1)] 
$\V$ is a finite set of vectors in $\C^n$.
\item[(2)]
$\B=(B_1,\ldots,B_k)$ where $k$ is odd,
and for all $i=1,\ldots,k$
we have that $B_i$ is an orthogonal basis of $\C^n$
and $B_i\subset \V$.
\item[(3)] 
For each $v\in\V$ the number of $i$ such that $v\in B_i$ is even.
\end{itemize}
\end{definition}

It is quite common in the literature 
\cite{T7-experiment,Lis-PRA,WA-preprint}
to refer to a KS pair 
as a {\em KS set,} and we will do so sometimes in this paper.

An extensive and recent summary of known examples of KS sets
in low dimensions is presented in \cite{Pavi}.
It turns out that until recent time,
the vast majority of known examples
have been found by  computer search,
however without much insight in the sets generated  \cite{Pavi}.
More recently, first computer-free constructions
have appeared that relate KS sets to some other
mathematical structures such as Hadamard matrices
\cite{Lis-Had}.
In this paper we continue this trend by giving
a simple, computer-free construction
of an infinite family of KS sets in $\R^4$.
This is the first time that an infinite family
of inequivalent KS sets in a space of fixed dimension is found.
Of course, an infinite continuous family of KS sets can be trivially
	constructed from one given KS set by applying  
	unitary transformations to it. We consider the KS sets
	thus obtained as {\em equivalent.}
	Our construction produces an infinite family
	of {\em inequivalent} KS sets; in particular the number of rays of our KS sets can attain any value that can be
	expressed as a product of two odd and relatively prime
	integers both greater than or equal to 3 (see Theorem 
	\ref{thm-main} below).

Moreover, four is the smallest possible dimension
in which KS sets 
described in Definition~\ref{def-KS-pair} can exist,
since the definition clearly requires the dimension to be even,
and the KS~theorem only holds in dimension at least~3.

KS sets
are key tools for proving some fundamental results in quantum theory 
and they also have 
various  applications in quantum information processing, see
\cite{T7-experiment,Lis-PRA} and the references therein.

\section{The new construction}
\label{sec-constr}

For integers $m>0$ and $s\ge 0$ we define the matrix
\begin{equation}
\label{eq-Rms}
R_{m,s}=\left(
\begin{array}{cc}
\cos\left(\frac{2\pi s}{m}\right)
& -\sin\left(\frac{2\pi s}{m}\right) \\
\sin\left(\frac{2\pi s}{m}\right)
&
\cos\left(\frac{2\pi s}{m}\right)
\end{array}
\right).
\end{equation}


\def\cv#1#2{\left(
\begin{array}{c}
#1\\
#2
\end{array}
\right)
}

\def\rv#1#2{\left( #1,\: #2 \right) }

\def\ot{\otimes}


By $\otimes$ we denote the Kronecker product of matrices.
We now state our main result.

\begin{theorem}
\label{thm-main}
Let $p,q\ge 3$ be relatively prime odd integers,
let $k_p,k_q$ be integers relative prime
to $p$ and $q$ respectively, and let $c$ be a non-zero real number
such that 
\begin{equation}
\label{eq-csq}
c^2
=
-\frac{\cos\left(2\pi \left(\frac{k_p}{p}-\frac{k_q}{q}\right)\right)}
{\cos\left(2\pi \left(\frac{k_p}{p}+\frac{k_q}{q}\right)\right)}.
\end{equation}

Let $a,b$ be vectors in $\R^4$ defined by
\begin{equation}
\label{eq-a-b}
a=\left(
\begin{array}{c}
(1-c)\cos\frac{2\pi k_q}{q}\\
(1-c)\sin\frac{2\pi k_q}{q}\\
-(1+c)\sin\frac{2\pi k_q}{q}\\
(1+c)\cos\frac{2\pi k_q}{q}
\end{array}
\right),
\ \ \ 
b=\left(
\begin{array}{c}
c+1 \\
0 \\
0 \\
c-1
\end{array}
\right).
\end{equation}
Let $r$ be the unique integer in the interval $(0,pq)$ such that
$r\equiv 1\pmod{p}$, $r\equiv -1\pmod{q}$. Let $M=R_{p,k_p}\ot R_{q,k_q}$.
We have:

(i) For $0\le i < pq$ the set $B_i=\{ M^ia, M^{i-1}a, M^ib, M^{i-r}b \}$
is an orthogonal basis of $\R^4$.

(ii) Let $\V=\{ M^ia \::\: 0\le i<pq \}\cup\{ M^ib \::\: 0\le i<pq \}$
and let $\B=\{ B_i \::\: 0\le i<pq \}$.
Then $(\V,\B)$ is a Kochen-Specker pair.
\end{theorem}
\begin{proof}
		The condition that $p$ and $q$ are relatively prime is used
		to guarantee the existence of the integer $r$ with the given properties,
		using the Chinese remainder theorem.
	
(i) 
The fact that $B_0$ is an orthogonal basis of $\R^4$
is proved
in Lemma~\ref{lem-B0-orth} below. 
Since $M$ is orthogonal, it follows that $B_i$
is an orthogonal basis of $\R^4$ for each $i$,
since $B_i$ is the image of $B_0$ under $M^i$.

(ii)
Let the indices of $B$ be taken modulo~$pq$.
For each $i$,
the vector $M^ia$ belongs to two bases, namely $B_i$ and $B_{i+1}$,
and
the vector $M^ib$ belongs to two bases, namely $B_i$ and $B_{i+r}$.
In Proposition~\ref{prop-vec-distinct} below we show
that the vectors in the set $\V$ are pairwise linearly independent.
Hence each element of $\V$ belongs to exactly
two bases in $\B$, and all conditions for Kochen-Specker pair are satisfied.
\end{proof}

We note that for given $p$ and $q$
it is always possible to choose $k_p$ and $k_q$
such that $c$ is a non-zero real number. There are many possible choices;
one of them is
\begin{equation}
\label{eq-kx}
				k_x =
\begin{cases} 
\left\lceil\frac{x}{4}\right\rceil \mbox{\ if\ }  x\equiv3\pmod4\\
\left\lfloor\frac{x}{4}\right\rfloor \mbox{\ if\ }  x\equiv1\pmod4.
\end{cases}
\end{equation}
Then it is easy to show that
\[
				\frac{\pi}{2} < 2\pi\left(\frac{k_p}{p}+\frac{k_q}{q}\right) < \frac{3\pi}{2}
\mbox{\ \ \ and \ \ } -\frac{\pi}{2} < 2\pi\left(\frac{k_p}{p}-\frac{k_q}{q}\right) < \frac{\pi}{2}
\]
hence
\[
-\frac{\cos\left(2\pi \left(\frac{k_p}{p}-\frac{k_q}{q}\right)\right)}
{\cos\left(2\pi \left(\frac{k_p}{p}+\frac{k_q}{q}\right)\right)}
>0.
\]

We will now work towards proving that $B_0$ is an orthogonal basis.
{\em For simplicity we will write just $R_p$ 
and $R_q$ instead of $R_{p,k_p}$ and $R_{q,k_q}$, respectively.}
To reduce the use of brackets in upcoming calculations
we will assume throughout that the ordinary matrix
product has a higher precedence 
among algebraic operations than the Kronecker product.
For example, the notation $X\otimes YZ$ denotes $X\otimes(YZ)$.

We note that the vectors $a$ and $b$ defined in Theorem \ref{thm-main}
can be written as
\[
a=
\cv{1-c}{0}\ot R_q\cv{1}{0}
+
\cv{0}{1+c}\ot R_q\cv{0}{1}
\]
and
\[
b=
\cv{c+1}{0}\ot \cv{1}{0}
+
\cv{0}{c-1}\ot \cv{0}{1}.
\]

\begin{lemma}
\label{lem-k0-l1}
Assume that $m=0$ or $n=1$.
Then
\[
a^T(R_p^m\ot R_q^{n})b=0.
\]
\end{lemma}
\begin{proof}
Since the matrices $R_p$ and $R_q$ are orthogonal, we have
\begin{eqnarray*}
a^T(R_p^m\ot R_q^{n})b
&=&
\left[ 
\rv{1-c}{0}\ot \rv{1}{0}R_q^{-1}
+\rv{0}{1+c}\ot \rv{0}{1}R_q^{-1}
\right]\times 
\\
&&
(R_p^m\ot R_q^{n})
\left[
\cv{c+1}{0}\ot\cv{1}{0} + \cv{0}{c-1}\ot \cv{0}{1}
\right].
\end{eqnarray*}
After expanding the right-hand side 
into four terms
and simplifying
each of them as
$(T\ot U)(V\ot W)(X\ot Y)=TVX \ot UWY$ 
we get
\begin{eqnarray*}
a^T(R_p^m\ot R_q^{n})b
&=&
\rv{1-c}{0}R_p^m\cv{c+1}{0} \ot \rv{1}{0}R_q^{n-1}\cv{1}{0}
\\
&&
+ \rv{0}{1+c}R_p^m\cv{c+1}{0} \ot \rv{0}{1}R_q^{n-1}\cv{1}{0}
\\
&&
+ \rv{1-c}{0}R_p^m\cv{0}{c-1} \ot \rv{1}{0}R_q^{n-1}\cv{0}{1}
\\
&&
+ \rv{0}{1+c}R_p^m\cv{0}{c-1} \ot \rv{0}{1}R_q^{n-1}\cv{0}{1}.
\end{eqnarray*}
If $m=0$ then
\[
 \rv{0}{1+c}R_p^m\cv{c+1}{0} =  \rv{1-c}{0}R_p^m\cv{0}{c-1} = 0
\]
while if $n=1$ then
\[
\rv{0}{1}R_q^{n-1}\cv{1}{0} = \rv{1}{0}R_q^{n-1}\cv{0}{1} =0.
\]
Hence if $m=0$ or $n=1$ then
\begin{eqnarray*}
a^T(R_p^m\ot R_q^{n})b
=
(1-c^2)(R_p^m)_{1,1}(R_q^{n-1})_{1,1}
+
(c^2-1)(R_p^m)_{2,2}(R_q^{n-1})_{2,2}=0
\end{eqnarray*}
because $(R_s^u)_{1,1}=(R_s^u)_{2,2}$ for all $s,u$.
\end{proof}

\begin{lemma}
\label{lem-B0-orth} 
The set $B_0$ defined in Theorem~\ref{thm-main}
is an orthogonal basis of $\R^4$.
\end{lemma}
\begin{proof}
The matrix $M$ is orthogonal,
hence $M^j=(M^{-j})^T$ for all $j$.
Note that $M^{-r}=R_p^{-1}\ot R_q$ by the properties of $r$.
The dot product of $b$ and $M^{-1}a$ can be written as
$a^TMb$.
Furthermore the dot product of $M^{i-1}a$ and $M^{i-r}b$ can be written as
\[
a^TM^{1-i}M^{i-r}b
=
a^TM^{1-r}b
=a^T(I\ot R_q^2)b.
\]
Therefore we need to prove the following six equalities,
which we write in a uniform way that allows us
to treat four of the six cases simultaneously.
			\begin{enumerate} 
				\item[(i)]   $a^T\left(I\otimes I\right)b = 0$
				\item[(ii)]  $a^T\left(R_{p}^{-1}\otimes R_{q}^{-1}\right)a = 0$
				\item[(iii)] $b^T\left(R_{p}^{-1}\otimes R_{q}\right)b = 0$
				\item[(iv)]  $a^T\left(R_{p}^{-1}\otimes R_{q}\right)b = 0$
				\item[(v)]   $a^T\left(R_{p}\otimes R_{q}\right)b = 0$
				\item[(vi)]  $a^T\left(I\otimes R_{q}^2\right)b = 0$.
			\end{enumerate}

Equalities (i), (iv), (v) and (vi) follow from Lemma~\ref{lem-k0-l1}.

In case (ii) we get
\begin{eqnarray*}
a^T(R_p^{-1}\ot R_q^{-1})a
&=&
\left[ 
\rv{1-c}{0}\ot \rv{1}{0}R_q^{-1}
+\rv{0}{1+c}\ot \rv{0}{1}R_q^{-1}
\right]
\times 
\\
&&
(R_p^{-1}\ot R_q^{-1})
\left[
\cv{1-c}{0}\ot R_q\cv{1}{0}
+
\cv{0}{1+c}\ot R_q\cv{0}{1}
\right]
\\
&=&
\rv{1-c}{0}R_p^T\cv{1-c}{0} \ot \rv{1}{0}R_q^{T}\cv{1}{0}
\\
&&
+ \rv{0}{1+c}R_p^T\cv{1-c}{0} \ot \rv{0}{1}R_q^{T}\cv{1}{0}
\\
&&
+ \rv{1-c}{0}R_p^T\cv{0}{1+c} \ot \rv{1}{0}R_q^{T}\cv{0}{1}
\\
&&
+ \rv{0}{1+c}R_p^T\cv{0}{1+c} \ot \rv{0}{1}R_q^{T}\cv{0}{1}
\\
&=&
(1-c)^2(R_p)_{1,1}(R_q)_{1,1}
+(1-c^2)(R_p)_{1,2}(R_q)_{1,2}\\
&&
+(1-c^2)(R_p)_{2,1}(R_q)_{2,1}
+(1+c)^2(R_p)_{2,2}(R_q)_{2,2}.
\end{eqnarray*}
Since 
$(R_s)_{1,1}=(R_s)_{2,2}$ 
and
$(R_s)_{1,2}=-(R_s)_{2,1}$ for each $s$,
the last expression simplifies to
\begin{equation}
\label{eq-case-ii}
(2c^2+2)(R_p)_{1,1}(R_q)_{1,1}+(2-2c^2)(R_p)_{1,2}(R_q)_{1,2}.
\end{equation}
This is equal to 0 exactly when
\begin{eqnarray}
c^2&=&-\frac{(R_p)_{1,1}(R_q)_{1,1}+(R_p)_{1,2}(R_q)_{1,2}}
{(R_p)_{1,1}(R_q)_{1,1}-(R_p)_{1,2}(R_q)_{1,2}}
\nonumber
\\
&=&
-\frac{\cos\left(\frac{2\pi k_p}{p}\right)\cos\left(\frac{2\pi k_q}{q}\right)
+\sin\left(\frac{2\pi k_p}{p}\right)\sin\left(\frac{2\pi k_q}{q}\right)}
{\cos\left(\frac{2\pi k_p}{p}\right)\cos\left(\frac{2\pi k_q}{q}\right)
-\sin\left(\frac{2\pi k_p}{p}\right)\sin\left(\frac{2\pi k_q}{q}\right)}
\nonumber
\\
&=&
-\frac{\cos\left(2\pi \left(\frac{k_p}{p}-\frac{k_q}{q}\right)\right)}
{\cos\left(2\pi \left(\frac{k_p}{p}+\frac{k_q}{q}\right)\right)}.
\label{eq-c-squared}
\end{eqnarray}

In case (iii) we get
\begin{eqnarray*}
b^T(R_p^{-1}\ot R_q)b
&=&
\left[ 
\rv{c+1}{0}\ot \rv{1}{0}
+\rv{0}{c-1}\ot \rv{0}{1}
\right]
\times 
\\
&&
(R_p^{-1}\ot R_q)
\left[
\cv{c+1}{0}\ot \cv{1}{0}
+
\cv{0}{c-1}\ot \cv{0}{1}
\right]
\\
&=&
\rv{c+1}{0}R_p^T\cv{c+1}{0} \ot \rv{1}{0}R_q\cv{1}{0}
\\
&&
+ \rv{c+1}{0}R_p^T\cv{0}{c-1} \ot \rv{1}{0}R_q\cv{0}{1}
\\
&&
+ \rv{0}{c-1}R_p^T\cv{c+1}{0} \ot \rv{0}{1}R_q\cv{1}{0}
\\
&&
+ \rv{0}{c-1}R_p^T\cv{0}{c-1} \ot \rv{0}{1}R_q\cv{0}{1}
\\
&=&
(c+1)^2(R_p)_{1,1}(R_q)_{1,1}
+(c^2-1)(R_p)_{2,1}(R_q)_{1,2}\\
&&
+(c^2-1)(R_p)_{1,2}(R_q)_{2,1}
+(c-1)^2(R_p)_{2,2}(R_q)_{2,2}
\\
&=&
(2c^2+2)(R_p)_{1,1}(R_q)_{1,1}-2(c^2-1)(R_p)_{1,2}(R_q)_{1,2}
\end{eqnarray*}
which is equal to (\ref{eq-case-ii}).
Hence assuming that $c^2$ equals the expression (\ref{eq-c-squared}),
which is a
necessary and sufficient condition for equality (ii) to hold, 
also implies that equality (iii) holds.

This completes the proof of Lemma~\ref{lem-B0-orth} and
hence also the proof of Theorem~\ref{thm-main}.
\end{proof}

Next we show that the $2pq$ vectors
used in our construction are pairwise
linearly independent. Strictly speaking this is not required
for the proof of the Kochen-Specker pair property,
however it may be of interest for example in the physical
implementations of our construction.

\begin{proposition}
\label{prop-vec-distinct}
Let $M,a,b$ be as in Theorem~\ref{thm-main}.
The $2pq$ vectors\break
$M^ia\ (0\le i<pq)$ and $M^ib\ (0\le i<pq)$
are pairwise linearly independent.
\end{proposition}
\begin{proof}
For a positive integer $s$ let $\z_s=e^{2\pi\sqrt{-1}/s}$ denote
the primitive $s$-th root of unity in $\C$.
The eigenvalues of $R_{p,k_p}^s \otimes R_{q,k_q}^t$ 
are $\z_p^{\pm sk_p}\z_q^{\pm tk_q}$ with all four combinations of $\pm$
signs in the exponents. Since $\gcd(p,q)=\gcd(p,k_p)=\gcd(q,k_q)=1$
and $p,q$ are odd,
it follows that the eigenvalues $\z_p^{\pm sk_p}\z_q^{\pm tk_q}$
are not real unless $s=t=0$. It follows that $M^i$ does not have
a real eigenvector for $0<i<pq$, hence 
$M^ia$ and $M^ja$ are linearly independent whenever\break 
$i\not\equiv j\pmod{pq}$,
and likewise
$M^ib$ and $M^jb$ are linearly independent whenever $i\not\equiv j\pmod{pq}$.

It remains to consider 
the possibility
that $M^ia$ and $b$ are linearly dependent
for some~$i$. 
Then $(M^ia)_2=(M^ia)_3=0$. 
Let $M^{-i}=R_p^m\otimes R_q^{n}$.
Explicit calculations  give
\begin{eqnarray}
(M^ia)_2 
&=&
a^T(R_p^m\otimes R_q^{n})\left[\cv{1}{0}\otimes \cv{0}{1}\right]
\nonumber
\\
&=& (1-c)(R_p^m)_{1,1}(R_q^{n-1})_{1,2}+(1+c)(R_p^m)_{2,1}(R_q^{n-1})_{2,2}
\label{eq-Mia2}
\\
(M^ia)_3 
&=&
a^T(R_p^m\otimes R_q^{n})\left[\cv{0}{1}\otimes \cv{1}{0}\right]
\nonumber
\\
&=& -(1+c)(R_p^m)_{1,1}(R_q^{n-1})_{1,2}-(1-c)(R_p^m)_{2,1}(R_q^{n-1})_{2,2}.
\label{eq-Mia3}
\end{eqnarray}
Assume that the right-hand sides of 
(\ref{eq-Mia2}) and (\ref{eq-Mia3}) both equal zero,
and subtract the latter from the former.
This gives
$(R_p^m)_{1,1}(R_q^{n-1})_{1,2}
=
-(R_p^m)_{2,1}(R_q^{n-1})_{2,2}$.
After plugging this into (\ref{eq-Mia2}) and performing some
elementary manipulations we deduce that $c=0$ or
\begin{equation}
\label{eq-last}
(R_p^m)_{1,1}(R_q^{n-1})_{1,2}
=
(R_p^m)_{2,1}(R_q^{n-1})_{2,2}=0.
\end{equation}
Since we assume a choice of $k_p,k_q$ such that $c\neq 0$,
we must have (\ref{eq-last}).
Denote
\[
Z=\left(
\begin{array}{cc}
0& -1\\
1& 0
\end{array}
\right).
\]
A quick argument shows that (\ref{eq-last}) implies
that either both $R_p^m$ and $R_q^{n-1}$
are equal to $Z$ or $-Z$,
or they are both equal to $I$ or $-I$.
The first case is not possible, since $p$ and $q$ are odd.
So assume that both $R_p^m$ and $R_q^{n-1}$
are equal to $I$ or $-I$. Then a calculation similar to those above shows
that
\[
M^ia =
\left(
\begin{array}{c}
\pm(1-c)\\
0\\
0\\
\pm(1+c)
\end{array}
\right)
\]
where signs are to be taken consistently. In comparison
with equation (\ref{eq-a-b}) we see that if $M^ia$ and $b$
are linearly dependent, then $c\neq\pm 1$,
and
$\frac{1-c}{c+1}=\frac{1+c}{c-1}$.
The last equation has no real solution, and this completes the proof.
\end{proof}

\newpage
\section{KS pairs as line graphs}

We now explain a connection of the KS pairs
constructed in Section~\ref{sec-constr}
and certain line graphs.
The material in this section is not required for proving the correctness
of our construction. It rather serves as an illustration
of the construction. We also mention analogy
with two other significant KS pairs
discovered earlier,
which also can be represented by line graphs.
Background information for the concepts
of graph and line graph can be found for example
in \cite{vLW}, or in many other places.

Suppose that $G$ is a simple graph. Recall
	that the {\em line graph} of $G$, denoted $L(G)$,
	is defined as follows. The vertices of $L(G)$ are
	the edges of $G$, and two vertices of $L(G)$ are adjacent
	if and only if the corresponding edges of $G$ share an endpoint. Note that if $v$ is a vertex of $G$ of degree $d$,
then the $d$ edges incident with $v$ in $G$ induce
a clique (complete subgraph) on $d$ vertices in $L(G)$.

Let $G$ be a simple graph and suppose
	that each vertex is labelled with a nonzero vector in $\C^d$
	for some $d$, such that if two vertices are adjacent
	in $G$, then the corresponding vectors
	are orthogonal in $\C^d$. We say that
	such a labelling of vertices of $G$ is 
	an {\em orthogonal representation} of $G$ in $\C^d$.
	Note that if $v$ is a vertex of degree $d$ in
	a graph $H$
	and the line graph $L(H)$ has  
	an orthogonal representation in $\C^d$,
	then the labels of the $d$ vertices of $L(H)$
	that correspond to the $d$ edges incident with $v$ in $H$
	form an {\em orthogonal basis of $\C^d$.}
	
	It follows from the discussion above that
		 if $G$ is a graph such that $G$
		has an odd number of vertices and each vertex of $G$ has the same degree $d$ (we say $G$ is $d$-regular),
		and $L(G)$ has an orthogonal representation in $\C^d$,
		then {\em this orthogonal representation is a KS pair.}
			Two important KS pairs
	discovered in the past have this particular structure. These are the KS pair
	with 18 vectors and 9 bases in $\R^4$  discovered
	by Cabello \cite{Adan-18} 
	that has the smallest number of vectors
	among known KS pairs,
	and the KS pair
	with 21 vectors and 7 bases in $\C^6$  discovered
	by Lison\v ek et al.\ \cite{Lis-PRA}
	which has the smallest number
	of bases among known KS pairs.
	The former KS pair is an orthogonal representation
	of the line graph $L(P_9)$ where $P_9$ is
	the Paley graph on 9 vertices (see \cite[Chapter~21]{vLW}).
	The latter KS pair is an orthogonal representation
	of the line graph $L(K_7)$ where $K_7$ is
	the complete graph on 7 vertices (the graph in which
	each two distinct vertices are adjacent).

Motivated by the fact that some important
	KS pairs can be represented using line graphs
	of highly symmetric graphs, we set out to seek
	more examples in this form. We tried
	to computationally find orthogonal representations of graphs $L(G)$
	in $\C^d$, where $G$ is $d$-regular vertex-transitive graph, using the lists
	of vertex-transitive graphs \cite{trans}. (A graph $G$
	is vertex-transitive if for any pair of vertices
	$u,v$ there is an automorphism of $G$ that maps $u$ to $v$;
	edge-transitive graphs are defined analogously.)
	A family of examples emerged from this search, which
	we were able to construct analytically (computer-free);
	the results are presented in this paper. In particular
the presentation of our new KS pairs using line graphs
of so-called chordal rings is given in Section~\ref{sec-line-cr}
below.

\subsection{Line graphs of chordal rings}
\label{sec-line-cr}

For integers $n$ and $k$ 
such that $1<k<n-1$
let  the {\em chordal ring}
with parameters $n$ and $k$ be the graph with vertex set $\Z_n$
(integers modulo~$n$) and with edge set 
\[
\{\{a,a+1\}:a\in\Z_n\}\cup\{\{a,a+k\}:a\in\Z_n\}.
\]
We will denote the chordal ring
with parameters $n$ and $k$ by ${\rm CR}(n,k)$.
It is immediate that ${\rm CR}(n,k)$
	is isomorphic to ${\rm CR}(n,n-k)$.
	If $1<k\le n/2$ then we can visualize
${\rm CR}(n,k)$ as the $n$-cycle 
(``ring'')
to which
all chords connecting pairs of vertices at distance $k$ were added;
this justifies the term ``chordal ring.''
For illustration we show a drawing of {\rm CR}(15,4)
in Figure~\ref{fig-cr}.

\begin{figure}[h]
	\centering
	\begin{tikzpicture} 
	\draw [fill] (0*2.5,1*2.5) circle [radius=0.05];
	\node [above] at (0,2.5) {0};
	\draw [fill] (0.406737*2.5, 0.913545*2.5) circle [radius=0.05];
	\node [above] at (0.406737*2.5, 0.913545*2.5) {14};
	\draw [fill] (0.743145*2.5, 0.669131*2.5) circle [radius=0.05];
	\node [right] at (0.743145*2.5, 0.669131*2.75) {13};
	\draw [fill] (0.951057*2.5, 0.309017*2.5) circle [radius=0.05];
	\node [right] at (0.951057*2.5, 0.309017*2.75) {12};
	\draw [fill] (0.994522*2.5, -0.104528*2.5) circle [radius=0.05];
	\node [right] at (0.994522*2.5, -0.104528*2.75) {11};
	\draw [fill] (0.866025*2.5, -0.5*2.5) circle [radius=0.05];
	\node [right] at (0.866025*2.5, -0.5*2.75) {10};
	\draw [fill] (0.587785*2.5, -0.809017*2.5) circle [radius=0.05];
	\node [right] at (0.587785*2.5, -0.809017*2.75) {9};
	\draw [fill] (0.207912*2.5, -0.978148*2.5) circle [radius=0.05];
	\node [right] at (0.207912*2, -0.978148*2.75) {8};
	\draw [fill] (-0.207912*2.5, -0.978148*2.5) circle [radius=0.05];
	\node [left] at (-0.207912*2, -0.978148*2.75) {7};
	\draw [fill] (-0.587785*2.5, -0.809017*2.5) circle [radius=0.05];
	\node [left] at (-0.587785*2.5, -0.809017*2.75) {6};
	\draw [fill] (-0.866025*2.5, -0.5*2.5) circle [radius=0.05];
	\node [left] at (-0.866025*2.5, -0.5*2.75) {5};
	\draw [fill] (-0.994522*2.5, -0.104528*2.5) circle [radius=0.05];
	\node [left] at (-0.994522*2.5, -0.104528*2.75) {4};
	\draw [fill] (-0.951057*2.5, 0.309017*2.5) circle [radius=0.05];
	\node [left] at (-0.951057*2.5, 0.309017*2.75) {3};
	\draw [fill] (-0.743145*2.5, 0.669131*2.5) circle [radius=0.05];
	\node [left] at (-0.743145*2.5, 0.669131*2.75) {2};
	\draw [fill] (-0.406737*2.5, 0.913545*2.5) circle [radius=0.05];
	\node [above] at (-0.406737*2.75, 0.913545*2.5) {1};
	
	\draw (0*2.5,1*2.5) -- (0.406737*2.5, 0.913545*2.5);
	\draw (0.406737*2.5, 0.913545*2.5) -- (0.743145*2.5, 0.669131*2.5);
	\draw (0.743145*2.5, 0.669131*2.5) -- (0.951057*2.5, 0.309017*2.5);
	\draw (0.951057*2.5, 0.309017*2.5) -- (0.994522*2.5, -0.104528*2.5);
	\draw (0.994522*2.5, -0.104528*2.5) -- (0.866025*2.5, -0.5*2.5);
	\draw (0.866025*2.5, -0.5*2.5) -- (0.587785*2.5, -0.809017*2.5);
	\draw (0.587785*2.5, -0.809017*2.5) -- (0.207912*2.5, -0.978148*2.5);
	\draw (0.207912*2.5, -0.978148*2.5) -- (-0.207912*2.5, -0.978148*2.5);
	\draw (-0.207912*2.5, -0.978148*2.5) -- (-0.587785*2.5, -0.809017*2.5);
	\draw (-0.587785*2.5, -0.809017*2.5) -- (-0.866025*2.5, -0.5*2.5);
	\draw (-0.866025*2.5, -0.5*2.5) -- (-0.994522*2.5, -0.104528*2.5);
	\draw (-0.994522*2.5, -0.104528*2.5) -- (-0.951057*2.5, 0.309017*2.5);
	\draw (-0.951057*2.5, 0.309017*2.5) -- (-0.743145*2.5, 0.669131*2.5);
	\draw (-0.743145*2.5, 0.669131*2.5) -- (-0.406737*2.5, 0.913545*2.5);
	\draw (-0.406737*2.5, 0.913545*2.5) -- (0*2.5,1*2.5);
	
	\draw (0*2.5,1*2.5) -- (0.994522*2.5, -0.104528*2.5);
	\draw (0.406737*2.5, 0.913545*2.5) -- (0.866025*2.5, -0.5*2.5);
	\draw (0.743145*2.5, 0.669131*2.5) -- (0.587785*2.5, -0.809017*2.5);
	\draw (0.951057*2.5, 0.309017*2.5) -- (0.207912*2.5, -0.978148*2.5);
	\draw (0.994522*2.5, -0.104528*2.5) -- (-0.207912*2.5, -0.978148*2.5);
	\draw (0.866025*2.5, -0.5*2.5) -- (-0.587785*2.5, -0.809017*2.5);
	\draw (0.587785*2.5, -0.809017*2.5) -- (-0.866025*2.5, -0.5*2.5);
	\draw (0.207912*2.5, -0.978148*2.5) -- (-0.994522*2.5, -0.104528*2.5);
	\draw (-0.207912*2.5, -0.978148*2.5) -- (-0.951057*2.5, 0.309017*2.5);
	\draw (-0.587785*2.5, -0.809017*2.5) -- (-0.743145*2.5, 0.669131*2.5);
	\draw (-0.866025*2.5, -0.5*2.5) -- (-0.406737*2.5, 0.913545*2.5);
	\draw (-0.994522*2.5, -0.104528*2.5) -- (0*2.5,1*2.5);
	\draw (-0.951057*2.5, 0.309017*2.5) -- (0.406737*2.5, 0.913545*2.5);
	\draw (-0.743145*2.5, 0.669131*2.5) -- (0.743145*2.5, 0.669131*2.5);
	\draw (-0.406737*2.5, 0.913545*2.5) -- (0.951057*2.5, 0.309017*2.5);
	\end{tikzpicture}
	\caption{The graph $\rm{CR}(15,4)$.} \label{CR_15_4}
	\label{fig-cr}
\end{figure}

It can be seen
from the construction of the orthogonal bases $B_i$ in Theorem~\ref{thm-main} and from the arguments
given in the proof of this theorem
that 
the KS pair constructed in Theorem~\ref{thm-main}	
	is 
	an orthogonal representation of $L({\rm CR}(pq,r))$,
	the line graph of the chordal ring ${\rm CR}(pq,r)$.
	In particular the $pq$ vertices of ${\rm CR}(pq,r)$
	correspond to the $pq$ orthogonal bases $B_i$ ($0\le i<pq$),
	and for each vertex of ${\rm CR}(pq,r)$ 
	the four edges incident
	with that vertex correspond to the four vectors
	that form the corresponding orthogonal basis.

\subsection{Vertex transitivity}

It is of special interest in quantum information
theory to know that the graph $L({\rm CR}(pq,r))$ 
is {\em vertex transitive,}
assuming that $p$, $q$ and $r$ are as in Theorem~\ref{thm-main}.
It can be equivalently stated as follows.

\begin{proposition}
Let $p$, $q$ and $r$ be as in Theorem~\ref{thm-main}.
The graph ${\rm CR}(pq,r)$ is edge transitive.
\end{proposition}
\begin{proof}
Consider any two edges of ${\rm CR}(pq,r)$.
If they are both ``ring'' edges or they are both chords,
then there is a cyclic shift of the vertices of 
${\rm CR}(pq,r)$ which is an automorphism of ${\rm CR}(pq,r)$
and it maps one of the edges to the other edge.

We are left with the case when one of the edges is a ``ring'' edge
and the other edge is a chord. Without loss of generality assume
that the edges are $\{0,1\}$ and $\{0,r\}$.
By the assumption of Theorem~\ref{thm-main} we have
$r\equiv 1\pmod{p}$ and $r\equiv -1\pmod{q}$.
Therefore $r$ is relatively prime to $pq$ and $r^2\equiv 1\pmod{pq}$
and $r\not\equiv \pm 1\pmod{pq}$.
Therefore the mapping 
$\varphi:a\mapsto ra$ is a bijection from $\Z_{pq}$ to $\Z_{pq}$,
the vertex set of ${\rm CR}(pq,r)$.
Its action on ring edges is $\varphi(\{a,a+1\})=\{ra,ra+r\}$
hence each ring edge is mapped to a chord.
Its action on chords is $\varphi(\{a,a+r\})=\{ra,ra+r^2\}=\{ra,ra+1\}$
hence each chord is mapped to a ring edge. 
Therefore $\varphi$ is an automorphism of  ${\rm CR}(pq,r)$.
Since $\varphi(\{0,1\})=\{0,r\}$, it follows that
${\rm CR}(pq,r)$ is edge transitive. 
\end{proof}

\section{A numerical example}

At the suggestion of a referee we conclude
the paper with a numerical
example for our KS pair construction given in
Theorem \ref{thm-main}. This example serves
as an illustration only.

Recall that Theorem \ref{thm-main} requires 
	the number of bases of the KS pair
	to be written as the product of two relatively prime
	odd integers $p$ and $q$ both grater than or equal to $3$.
	Hence the smallest example has $15$ bases.
Let us take $p=3$ and $q=5$. 
According to Theorem \ref{thm-main}
the value of $r$ 
is the unique solution in
the interval $(0,15)$ to the system of congruences
\begin{eqnarray*}
r&\equiv& \phantom{-}1\pmod{3}\\
r&\equiv& -1\pmod{5}.
\end{eqnarray*}
This solution is $r=4$.

Therefore, as explained above,
the graph $\rm{CR}(15,4)$ in Figure \ref{fig-cr}
illustrates a KS pair with $pq=15$ orthogonal bases of $\R^4$
and $2pq=30$ vectors.
Each vertex corresponds to one basis
of the KS pair, and the four edges incident
with a vertex correspond to the four vectors
forming the basis corresponding to that vertex.

Applying equation (\ref{eq-kx})
	we get $k_p=\lceil \frac{3}{4} \rceil=1$
	and $k_q=\lfloor \frac{5}{4} \rfloor=1$. 
From now on we will round all 
numerical values obtained in 
forthcoming computations to six decimal places.
From equation (\ref{eq-csq})
we get
\[
c
=\sqrt{
-\frac{\cos\left(2\pi \left(\frac{1}{3}-\frac{1}{5}\right)\right)}
{\cos\left(2\pi \left(\frac{1}{3}+\frac{1}{5}\right)\right)}
}
=0.827091.
\]
Plugging into (\ref{eq-a-b}) we get
\[
a=\left(
\begin{array}{c}
0.053432\\
0.164446\\
-1.737667\\
0.564602
\end{array}\right)
\mbox{\ \ \ and\ \ \ }
b=\left(
\begin{array}{c}
1.827091\\
0\\
0\\
-0.172909
\end{array}\right).
\]
From equation (\ref{eq-Rms}) we get
\begin{eqnarray*}
R_{p,k_p}&=&R_{3,1}=
\left(
\begin{array}{cc}
	\cos\left(\frac{2\pi}{3}\right)
	& -\sin\left(\frac{2\pi}{3}\right) \\
	\sin\left(\frac{2\pi}{3}\right)
	&
	\cos\left(\frac{2\pi}{3}\right)
\end{array}
\right)
\\
R_{q,k_q}&=&R_{5,1}=
\left(
\begin{array}{cc}
	\cos\left(\frac{2\pi}{5}\right)
	& -\sin\left(\frac{2\pi}{5}\right) \\
	\sin\left(\frac{2\pi}{5}\right)
	&
	\cos\left(\frac{2\pi}{5}\right)
\end{array}
\right)
\end{eqnarray*}
and the matrix $M$ defined in Theorem \ref{thm-main}
is
\[
M = R_{3,1} \otimes R_{5,1}
=
\left(
\begin{array}{cccc}
-0.154508 & 0.475528 & -0.267617 & 0.823639 \\ 
-0.475528 & -0.154508 & -0.823639 & -0.267617 \\
0.267617 & -0.823639 & -0.154508 & 0.475528 \\ 
0.823639 & 0.267617 & -0.475528 & -0.154508
\end{array}
\right).
\]
With $r$, $a$, $b$ and $M$  determined, all vectors and bases
of the KS pair can now be computed easily.
Recall from Theorem \ref{thm-main}
that there are $2pq$ vectors in total, namely
$M^ia$ and $M^ib$ where $0\le i<pq$, and $pq$ orthogonal bases
$B_i$ where $0\le i<pq$, and each basis $B_i$ consists
of vectors $M^ia$, $M^{i-1}a$, $M^ib$ and $M^{i-r}b$.
We will finish this example by computing the vectors in 
one of the bases. Let us take, for example, $i=7$.
Then 
\[
B_7=\{M^7a, M^6a, M^7b, M^3b\}
\]
and the explicit coordinates of the vectors in $B_7$ are
\begin{eqnarray*}
M^7a=
\left(
\begin{array}{c}
-0.860114\\
1.330930\\ -0.658114\\ 0.651057
\end{array}
\right),\ \ \
&
M^6a=
\left(
\begin{array}{c}
-0.139886\\ 0.101633 \\ -1.073937 \\ -1.478147
\end{array}
\right),
\\
M^7b=
\left(
\begin{array}{c}
0.651057 \\ -0.658114\\ -1.330930\\ 0.860114
\end{array}
\right),\ \ \
&
M^3b=
\left(
\begin{array}{c}
-1.478148\\ -1.073937\\ -0.101633\\ 0.139886
\end{array}
\right).
\end{eqnarray*}
It can be checked numerically that the dot product
of any two distinct vectors in $B_7$
is zero up to a rounding error. This concludes the example.

\section*{Acknowledgement}
Research of both authors was supported
in part by the Natural Sciences and Engineering Research Council of Canada
(NSERC).


\begin{thebibliography}{10}


\bibitem{Adan-18}
A.~Cabello, A proof with 18 vectors of the Bell-Kochen-Specker theorem. 
In: M.~Ferrero and A.~van der Merwe (Eds.), 
New Developments on Fundamental Problems in Quantum Physics. 
Kluwer Academic, Dordrecht, Holland, 1997, pp.~59--62.

\bibitem{T7-experiment}
G.~Ca\~{n}as, M.~Arias, 
S.~Etcheverry, 
E.S.~G\'{o}mez, A.~Cabello, G.B.~Xavier, G.~Lima,
Applying the simplest Kochen-Specker set for quantum information processing.
Phys.\ Rev.\ Lett.\ {\bf 113} (2014), 090404.


\bibitem{How}
M.~Howard, J.~Wallman, V.~Veitch, J.~Emerson,
Contextuality supplies the `magic' for quantum computation.
Nature {\bf 510} (2014), 351--355.


\bibitem{KS-paper}
S.~Kochen, E.P.~Specker, 
The problem of hidden variables in quantum mechanics. 
Journal of Mathematics and Mechanics {\bf 17} (1967), 59--87.


\bibitem{vLW}
J.H.~van Lint, R.M.~Wilson, A Course in Combinatorics.
Second Edition. Cambridge University Press, 2001.


\bibitem{Lis-PRA}
P.~Lison\v{e}k, P.~Badzi\c{a}g, J.R.~Portillo, A. Cabello,
Kochen-Specker set with seven contexts. 
Phys. Rev. A {\bf 89} (2014), 042101. 

\bibitem{Lis-Had}
P.~Lison\v{e}k, 
Kochen-Specker sets and Hadamard matrices. 
Theoretical Computer Science {\bf 800} (2019), 142--145. 


\bibitem{Pavi}
M.~Pavi\v{c}i\'{c},
Arbitrarily exhaustive hypergraph generation of 4-, 6-,\break 
8-, 16-, and 32-dimensional quantum contextual sets.
Phys. Rev. A {\bf 95} (2017), 062121.

\bibitem{trans}
G.~Royle, D.~Holt,
Vertex-transitive graphs on fewer than 48 vertices.
{\tt https://zenodo.org/records/4010122}


\bibitem{WA-preprint}
M.~Waegell, P.K.~Aravind,
The minimum complexity of Kochen-Specker sets does not scale with dimension.
Phys. Rev. A {\bf 95} (2017), 050101.

\end{thebibliography}
\end{document}